\documentclass[a4paper, 10pt]{article}

\usepackage{amsthm}
\usepackage{amsmath}
\usepackage{amssymb}
\usepackage{enumerate}
\usepackage[mathscr]{eucal}

\let\quadr\Quadr
\def\collin{\sim}
\def\oadjac{\mathrel{\mbox{$\collin\mkern-18mu{\raise0.8ex\hbox{$\perp$}}$}}}

\def\toadjac{\mathrel{\lower.35ex\hbox{\baselineskip-3pt\lineskip-3pt\vbox{\hbox{$\sim$}\hbox{$\sim$}}}}}
\let\doadjac\toadjac
\def\soadjac{\lower.25ex\hbox{\scriptsize\baselineskip-2.2pt\lineskip-2.2pt\vbox{\hbox{$\sim$}\hbox{$\sim$}}}}
\def\ssoadjac{\lower.21ex\hbox{\tiny\baselineskip-2.2pt\lineskip-2.2pt\vbox{\hbox{$\sim$}\hbox{$\sim$}}}}
\def\badjac{\mathchoice{\doadjac}{\toadjac}{\soadjac}{\ssoadjac}}

\def\lines{{\cal L}}
\def\izolines{{\cal G}}
\let\projlines\izolines
\def\aflines{{\cal A}}
\def\cykle{{\cal C}}
\def\inc{\mathrel{\rule{2pt}{0pt}\rule{1pt}{9pt}\rule{2pt}{0pt}}}
\def\pek{{\bf p}}
\def\wiazka{{\bf s}}
\def\peki{{\cal P}}

\def\propafpeki{\peki_{\aflines}^{{\sf 1a}}}  
\def\conicpeki{\peki_{\izolines}^{{\sf c}}}
\def\projpeki{\peki_{\izolines}^{{\sf p}}}
\def\propsafpeki{\peki_{\izolines}^{{\sf sa}}}
\def\cylinpeki{\peki_{\aflines}^{{\sf l}}} 
\def\parsafpeki{\peki_{\aflines}^{{\sf sa}}} 
\def\parafpeki{\peki_{\aflines}^{{\sf a}}}

\def\topof{{\mathrm{T}}}

\def\GrasSpace(#1,#2){{{\bf G}_{#1}({#2})}}

\def\fixlag{\ensuremath{\goth L}}
\def\fixproj{\ensuremath{\goth P}}

\def\afdim{{\mbox{\boldmath$a$}}}
\def\basdim{{\mbox{\boldmath$\nu$}}}
\def\gendim{{\mbox{\boldmath$t$}}}

\newenvironment{zapo}{%
\begin{center}\begin{minipage}[m]{0.7\textwidth}\em}{\end{minipage}\end{center}}

\newenvironment{ctext}{%
  \par
  \smallskip
  \centering
}{%
 \par
 \smallskip
 \csname @endpetrue\endcsname
}

\let\goth\mathfrak
\def\struct#1{{\ensuremath{\langle #1 \rangle}}}
\def\sub{\raise.5ex\hbox{\ensuremath{\wp}}}
\def\Lor{\;\lor\;}
\def\Land{\;\land\;}

\def\lemmaname{Lemma}
\def\theoremname{Theorem}
\def\propositionname{Proposition}
\def\corollaryname{Corollary}
\def\factname{Fact}

\newtheorem{theorem}{\theoremname}[section]
\newtheorem{thm}[theorem]{\theoremname}
\newtheorem{prop}[theorem]{\propositionname}
\newtheorem{lem}[theorem]{\lemmaname}
\newtheorem{cor}[theorem]{\corollaryname}
\newtheorem{fact}[theorem]{\factname}

\begin{document}

\title{Pencils of lines in generalized Laguerre spaces}
\author{Krzysztof Radziszewski }

\maketitle

\begin{abstract}
 In the paper we characterize subspaces and pencils of lines of 
 generalized (and multidimensional) Laguerre spaces and
 we consider definability of the structure of "conic" pencils 
 in the Grassmann space of 1-subspaces. 
 We also study definability of the underlying Laguerre geometry 
 in terms of structures of pencils of lines
 for some, more interesting, systems of pencils.
\vskip3pt 
\par\noindent
 keywords: {Multidimensional Laguerre spaces, subspaces, pencils, Grassmann spaces}
\end{abstract}


\section{Introduction}

The classical Laguerre Geometry originates, primarily, in the geometry
of {\em Spee\-ren} (oriented lines) and {\em Zykeln} (oriented circles) of a Euclidean plane
which, after the Blaschke transformation can be visualized as the structure
of conics on  a projective cylinder, cf. \cite{blaszke}.
This definition can be generalized in various directions: one can investigate 
structures of oriented hyperplanes (cf. \cite{gusc}),
structures of oriented lines in a higher dimensional Euclidean space (cf. \cite{cycl3eucl}),
ovals and ovoidal surfaces replacing the ``basic'' conic of a cylinder (cf. \cite{barlot:strambach},
\cite{benc:majer}, \cite{majer}), structures on degenerate ``quadrics'' (cf. \cite{degpolar}). 

Passing to a, possibly one of the widest generalizations, we can roughly say that
{\em a generalized Laguerre space} is a geometrical structure based on a projective 
cone (with its vertex deleted) $S$ defined over a nondegenerate quadric. Several attempts to geometry on $S$
are possible. One can look, primarily, at generators of $S$, considering the underlying
point universe as the set of self conjugate points of a degenerate polarity. Following this approach
one enters into the world of polar geometry (see e.g. \cite{camer}). 
One can also try to imitate the approach of chain geometry
(cf. e.g. \cite{benz}, \cite{herc})
and distinguish as primitives the family of conics on $S$. In this paper we are closer to this second
tradition. However, if a chain space contains lines (definable in terms of chains, as it happens e.g.
in Benz-Minkowski planes, cf. \cite{benz}), it is much more convenient to have these lines distinguished as individuals
of some other sort. So, finally, a structure under our consideration has form 
$\fixlag = \struct{\text{points}, \text{cycles (=chains)}, \text{lines}}$.

With each geometrical structure we can associate the family of its subspaces.
In our case we can define the dimension function on the subspaces of \fixlag\
and after that the standard construction of the Grassmann space of $k$-subspaces of
\fixlag\ can be applied and pencils of these subspaces can be defined.
Then (also a standard one) question in the spirit of Chow appears:
can we recover the underlying geometry in terms of a geometry on its $k$-subspaces.

In the paper we answer the above question in a very particular case $k=1$; moreover, 
we restrict ourselves to geometry of lines of \fixlag.
Practically, we study in some detail the Grassmannian of (all) 1-dimensional subspaces and 
structures of pencils of lines. Dealing with (projectively) planar pencils leads to 
arguments from polar geometry, so we only mention these pencils at the end of the paper
and we concentrate upon pencils determined by (projective) cones.
The structure of such pencils seems interesting also on its own right: they do not yield any
partial linear space but, instead, they introduce a chain-space-like structure.
In both considered cases the answer is affirmative i.e. the structure \fixlag\ can
be recovered in terms of respective structures of lines.

Questions concerning structures defined on the set of cycles are addressed in some other papers.
In case of Grassmannians of chains one can, perhaps, apply techniques of \cite{pamb} and
\cite{michalak}. In case of pencils of chains (especially when so called tangent pencils are
involved) some troubles appear concerning tangency classes of chains, 
so that world needs other techniques.
\newline
As $k$ increases, \fixlag\ admits $k$-subspaces that carry quite different geometries 
and investigations on Grassmannians defined on them become much more complex.

\subsection{Definitions}

Now, let us make the geometry considered in the paper more precise.
Let $\goth P$ be a finite dimensional Pappian and non-Fanoian
projective space.
Next, let $T,B$ be two transversal subspaces of \fixproj\ and 
$\quadr$ be a nondegenerate quadric on $B$.
Finally, let $S_0$ be a projective cone with vertex $T$ over a quadric $\quadr$, 
contained in $\goth P$,
and let $S = S_0 \setminus T$.
Let $\lines$ be the set of all (nonempty) sections with $S$ of
the lines of $\goth P$ which lie on $S_0$.
Then $\lines = \aflines \cup \izolines$,
where 
$\aflines$ consists of the (sections of) lines meeting $T$, and 
$\izolines$ consists of the lines on $S_0$ that miss $T$; each line in $\izolines$ is 
contained in a subspace of the form $M+T$, where $M$ is a generator of $\quadr$.
Elements of $\izolines$ are called {\em projective lines}, and the elements of $\aflines$
are {\em affine lines}.
In the family $\aflines$ we have a natural parallelism $\parallel$. 
Let $\cykle$ be the family of all nontrivial sections $S \cap A$
which do not contain a line, where $A$ is a plane of $\goth P$.
Then $\cykle$ consists of all the conics on $S_0$. 
Set $\inc \;\; = \;\; \in \;\; \subset S\times(\lines\cup\cykle)$.
Finally,   
  $$ {\fixlag} = \struct{S,\cykle,\izolines,\aflines,\parallel,\inc} $$ 
is a generalized Laguerre space defined over a ruled quadric $\quadr$ 
and contained in a projective space $\goth P$.
If  $\quadr$ is a nonruled quadric then $\izolines = \emptyset$, $\lines = \aflines$,
and most of subsequent results concern ``nothing'' (structures with the void universe).
Thus in the whole paper {\em we assume that  $\quadr$ is ruled}.

\section{Subspaces}\label{sec:peki}

A subspace of \fixlag\ is a subset $X \subset S$ such that the conditions:
\begin{itemize}\def\labelitemi{--}
  \item
    if $|K\cap X|\geq 2$, then $K\subset X$,
  \item
    if $a\in X\cap K$ and $K\parallel M\subset X$, then $K\subset X$,
  \item
    if $|A\cap X|\geq 3$, then $A\subset X$,
  \item
    if $|X\cap A|\geq 2$ and $A\inc B\subset X$, then $A\subset X$ 
  
  ($\inc$ is the relation of tangency of cycles in this case )
\end{itemize}
hold for every $a\in S$, $ K,M\in \lines$, $ A,B\in \cykle$.

From some point of view it is relatively easy to characterize subspaces of \fixlag;
in the projective representation of the Laguerre space $\fixlag$ we have 
started from, a subspace $X$ is a section of $S$ with a projective subspace 
$\overline{S}$.
On the other hand, subspaces 
may carry quite different geometries.
If $\dim(X) =1$ the situation is clear: $X$ is either a line 
(affine or projective, then $\dim({\overline{X}}) = 1$),
or a cycle ($\dim({\overline{X}}) = 2$).
Roughly speaking, in any case $X$ is a generalized Laguerre space or it is 
a projective, semiaffine or affine space.

\section{Classification of subspaces}

  In this section we shall give a detailed classification of the subspaces of 
  \fixlag\ and analyze the arising incidence geometry.

Let us write, generally $\sub(\fixlag)$ for the family of all the subspaces of \fixlag\ 
and  $\sub_k(\fixlag)$ for the family of $k$-dimensional subspaces of 
\fixlag.

\def\subcon{{\goth d}}
\def\sublag{{\goth l}}
\def\subgen{{\goth g}}
\def\submix{{\goth s}}

Next, let us write
\begin{description}\itemsep-2pt
\item[{$\submix^{m,m'}_{d,w}$}] for the set of all affine cones with $m$-dimensional
  affine generator defined over a projective cone with $m'$-dimensional projective 
  generator and basis being a $d$-dimensional quadric with index $w$
  (note that if $X\in\submix^{m,m'}_{d,w}$ then the dimension of the vertex of
  the ``projective part" of $X$ is $m'-w-1$);
\item[{$\sublag^m_{d,w}$}]  for the set of all affine cones with $m$-dimensional
  affine generator and the $d$-dimensional quadric with index $w$ as a basis
  contained in \fixlag;
\item[{$\subcon^m_{d,w}$}] for the set of all projective cones with $m$-dimensional
  projective generator (i.e.. with $(m-w-1)$-dimensional vertex)
  and the $d$-dimensional quadric with index $w$ as a basis
  contained in \fixlag; 
\item[{$\subgen^m_d$}] for the set of $(m + d)$-dimensional
  generators of \fixlag\ which maximal affine subgenerator
  has dimension $m$.
\end{description}
From the definitions we have 
\begin{itemize}\def\labelitemi{-}\itemsep-2pt
\item
  $\sublag^{m}_{d,w} = \submix^{m,w}_{d,w}$ and
  $\subcon^{m}_{d,w} = \submix^{0,m}_{d,w}$,
\item
  $\subgen^0_d$ is the set of all $d$-dimensional projective generators of \fixlag,
\item
  $\subgen_0^m$ is the set of all $m$-dimensional affine generators of \fixlag; in 
  particular
\item 
  $\subgen^0_1 = \izolines$, 
\item
  $\subgen^1_0 = \aflines$,
\item
  $\sublag^{0}_{d,w}$ ($0$ is the dimension of a point)
  is the set of all subquadrics contained in \fixlag; 
  formally, also $\sublag^0_{d,w} = \subcon^w_{d,w}$.
  Thus
\item
  $\sublag^{0}_{1,0} = \subcon^0_{1,0} = \cykle$.
\end{itemize}

Let us write
  $\basdim$ for the dimension of a base of \fixlag\ 
  i.e. $\basdim = \dim(B)+1$,
  $\gendim$ for the dimension of a maximal projective generator of \fixlag\ 
  so $\gendim$ is the index of $Q$, 
  and
  $\afdim$ for the dimension of a maximal affine generator of \fixlag\ 
  i.e. $\afdim = \dim(T)+1$.

Fundamental properties of the subspaces of $\fixlag$ are given in the 
subsequent \ref{subsp:gen1}--\ref{subsp:gen5}.
\begin{fact}\label{subsp:gen1}
  A maximal proper subspace of \fixlag\ is an element of 
  the following four sets:
  $\sublag^{\afdim}_{\basdim-1,\gendim-1}$, 
  $\sublag^{\afdim}_{\basdim-1,\gendim}$, 
  $\sublag^{\afdim-1}_{\basdim,\gendim}$, and
  $\submix^{\afdim,\gendim}_{\basdim-2,\gendim-1}$.
  \end{fact}
\begin{fact}\label{subsp:gen2}
  If $X \in \sublag^m_{d,w}$ then $\dim(X) = d+m$.
\end{fact}
\begin{fact}\label{subsp:gen3}
  If $\sub(\fixlag)\ni X \subset X'\in\sublag^m_{d,w}$ then 
  $X$ belongs to $\submix^{m_1,m'_1}_{d_1,w_1}$ or $X$ belongs to 
  $\subgen^{m_1}_{w_1}$ 
  where $\sublag^m_{d,w}\ne \submix^{m_1,m'_1}_{d_1,w_1}$,  
  $m_1=0,1,\ldots ,m, m'_1=0,1,\ldots ,w, d_1=0,1,\ldots ,d, 
  w_1=0,1,\ldots ,w$, and $\submix^{m_1,m'_1}_{d_1,w_1}$ is well defined.
\end{fact}

\begin{thm}\label{subsp:gen4}
  Let $X \in \submix^{m,m'}_{d,w}$. Then the geometry of the restriction
  $\fixlag\vert{X}$ of \fixlag \ to $X$ is the Laguerre space with 
  $\basdim_{\fixlag\vert{X}} = d+m'-w$,
  $\gendim_{\fixlag\vert{X}} = m'$, and
  $\afdim_{\fixlag\vert{X}} = m$.
\end{thm}
\begin{thm}\label{subsp:gen5}
  Let $X \in \subgen_{d}^{m}$. 
  Then the geometry of the restriction  $\fixlag\vert{X}$ of \fixlag \ to $X$ is 
  a semiaffine linear space (a hole space, c.f. \cite{sapls:1}, also called a slit space).
\end{thm}

\section{Pencils, general construction}

In accordance with the general approach adopted in incidence geometries
for an integer $k$ and $X,Z \in \sub(\fixlag)$ with $\dim(X) = k-1$, $\dim(Z) = k+1$
we define the $k$-pencil $\pek(X,Z)$ to be the set
\begin{equation}\label{def:pek}
  \pek(X,Z)\quad = \quad \left\{ Y\in\sub(\fixlag)\colon X \subset Y \subset Z,\;
  \dim(Y) = k\right\}.
\end{equation}
Actually, the formula \eqref{def:pek} defines too wide class of subsets
than those which are usually referred to as pencils.
For example, usually the set
of the cycles through a point on a (projective) sphere 
is not considered as a pencil.
At least two properties of a currently investigated family $\peki$
of pencils should be satisfied:
\begin{itemize}\itemsep-2pt\def\labelitemi{--}\em
\item
  $|p|\geq 3$ for any $p\in\peki${\normalfont;}
\item
  if $p_1,p_2 \in \peki$ and $p_1 \subset p_2$ then $p_1 = p_2$.
\end{itemize}
In the sequel in each particular case we shall write down explicitly
what types of pencils are currently admitted.

Let us begin with an analysis of possible $1$-pencils.
We write
%
$\wiazka(a)$ for the set of all lines through a point $a$ and we call such a set {\em a star}.
Analogously, one defines stars of cycles; in the paper pencils of cycles are not investigated though,
and therefore stars of cycles are not needed here.
Then suitable pencils have form
\begin{equation}\label{wz:genpenk} 
  \wiazka \cap \{ Y \colon Y \subset Z \} \text{ where }  Z\in\sub(\fixlag), \; \dim(Z) = 2,
  \text{ and } \wiazka \text{ is a star}.
\end{equation}
Two dimensional subspaces of \fixlag\ are the elements of the following classes:
%
  $\subgen^0_2$ (projective planes), 
  $\subgen^1_1$ (semiaffine planes), 
  $\subgen^2_0$ (affine planes),
  $\sublag^1_{1,0}$ (affine cones = Laguerre planes), 
  $\subcon^1_{1,0}$ (projective cones),
  $\sublag^0_{2,0}$ (M{\"o}bius planes), 
  $\sublag^0_{2,1}$ (Minkowski planes).
%
Note that the class $\submix^{1,1}_{0,0}$ is excluded from $\sub_2(\fixlag)$; its 
elements are somehow ``strange", as they are unions of two planes in $\subgen^1_1$.

%
Recall that we restrict ourselves to {\em pencils of  lines} only, and our pencils should be at least 
3-element sets. Consequently, it suffices to consider sets of the form \eqref{wz:genpenk}
with $\wiazka = \wiazka(a)$ and $a \in Z$
where $Z$ is one of the following:
an affine plane, 
a projective cone 
(in this case we assume, additionally, that $a$ is a vertex of $Z$; 
without this assumption the corresponding pencil would consist of one line only),
a projective plane, 
and
a semiaffine plane, resp.. 
The obtained classes of pencils are denoted by  
$\propafpeki$   ({\em proper affine pencils}),  
$\conicpeki$   ({\em conic pencils}),
$\projpeki$   ({\em proper projective pencils}), and
$\propsafpeki$   ({\em proper semiaffine pencils}).
Note that the elements of a conic pencil and of a proper projective pencil are projective lines,
the elements of a proper affine pencil are affine lines, while a proper semiaffine pencil
contains one affine line and its remaining lines are projective. In what follows we shall also consider 
restrictions of pencils in $\propsafpeki$ to projective lines and such a restriction will be also 
called a proper semiaffine pencil.

In case of the currently considered geometry we have the notion of a parallelism distinguished; 
in such a geometry so called ``parallel pencils'' are frequently considered. 
We follow this tradition and we consider sets of the form 
\begin{equation}\label{wz:genparpenk}
  \left\{ L \in\aflines\colon L\parallel L_0,\; L \subset Z \right\}, \text{ where }
  Z \in \sub(\fixlag), \; \dim(Z) = 2. 
\end{equation}
As above, it suffices to consider the cases when 
$Z$ is 
an affine cone, 
a semiaffine plane, and
an affine plane. 
The classes of pencils thus obtained are denoted as follows:
$\cylinpeki$ ({\em cylinder pencils}),
$\parsafpeki$ ({\em parallel semiaffine pencils}),
and
$\parafpeki$ ({\em parallel affine pencils}).

In what follows we shall try to define the underlying Laguerre geometry in terms
of the structures of the form
\begin{ctext}
  $\struct{\text{lines},\text{pencils}\_\text{of}\_\text{lines}}$
\end{ctext}
for some, more interesting, systems of pencils.
%

\section{Grassmann spaces and spaces of pencils associated with \fixlag}

One more notion will be used intensively in the sequel:
for $K_1,K_2\in \lines$ we write
\begin{eqnarray*}
   K_1 \badjac K_2 & \iff & K_1 \neq K_2 \Land \exists a \in S \; [a \inc K_1,K_2].
\end{eqnarray*}   
\subsection{Grassmann space of $1$-subspaces}

Let us begin with the simplest case when the points of the considered structure
are the $1$-dimensional subspaces of \fixlag. In symbols, 
\def\projplanes{\subgen^0_2} 
\def\safplanes{\subgen^1_1} 
\def\afplanes{\subgen^2_0} 
\def\afcones{\sublag^1_{1,0}} 
\def\projcones{\subcon^1_{1,0}} 
\def\mobplanes{\sublag^0_{2,0}} 
\def\minplanes{\sublag^0_{2,1}} 
\begin{ctext}
  $\sub_1(\fixlag) = \cykle \cup \lines$ \quad and \quad 
  $\sub_2(\fixlag) = 
  \projplanes \cup 
  \safplanes \cup 
  \afplanes \cup 
  \afcones \cup 
  \projcones \cup 
  \mobplanes \cup 
  \minplanes$.  
\end{ctext}
%

%
\begin{zapo}
  In what follows we shall be concerned with the structure
  \\
  \centerline{$\GrasSpace(1,\fixlag) := \struct{\sub_1(\fixlag),\sub_2(\fixlag),\subset}$.}
\end{zapo}
Let us write, for short, $\sub_1 = \sub_1(\fixlag)$ and $\sub_2 = \sub_2(\fixlag)$.
%
%
%
%

Through a series of subsequent lemmas we shall distinguish in terms of the 
geometry of $\GrasSpace(1,\fixlag)$ basic types of corresponding subspaces.

The crucial observation consists in the following lemma, which shows when
(formally considered) the fundamental axiom of partial linear spaces fails
in the structure $\GrasSpace(1,\fixlag)$.
\begin{lem}\label{lem:gras:2in2}
  Let $Y,Y'\in\sub_1$ and $Z,Z'\in\sub_2$. If 
  $Y,Y' \subset Z,Z'$ and $Y\neq Y'$, $Z\neq Z'$ then
  $Y,Y'\in\lines$ and $Z,Z'\in\afcones\cup\projcones\cup\minplanes$.
\end{lem}
\begin{proof}
  Suppose that $Y\in\cykle$; then $Y,Y'$ uniquely determines the subspace
  in $\sub_2$ which contains $Y,Y'$. This proves that $Y,Y'\notin\cykle$.
  Then possible $Z,Z'$ are those elements of $\sub_2$ that contain a line.
  On the other hand the planes of \fixlag, i.e. the elements of $\safplanes$,
  $\afplanes$, and $\projplanes$, if distinct may have at most a line in common.
  Moreover, no plane can cross a $2$-dimensional subspace that is not a plane in two lines.
  This finally yields our claim.
\end{proof}

Let us see that the lines $Y,Y'$ from \ref{lem:gras:2in2} are such that $Y\parallel Y'$ or 
$Y\badjac Y'$, because otherwise $Y,Y'$ uniquely determines the subspace in $\sub_2$ 
which contains $Y,Y'$.

As a direct consequence of \ref{lem:gras:2in2} we obtain
\begin{lem}\label{lem:gras2lines}
  Let $Y \in \sub_1$. Then
  \begin{equation}\label{def:gras2lines}
    Y \in \lines  \iff  (\exists {Y' \in \sub_1,\, Z,Z'\in\sub_2})
    [Y\neq Y' \land Z\neq Z' \land Y,Y' \subset Z,Z'].
  \end{equation}
  Consequently, the class $\lines$ is definable in terms of 
  $\GrasSpace(1,\fixlag)$.
\end{lem}
\begin{proof}
  The right-to-left implication of \eqref{def:gras2lines} follows from
  \ref{lem:gras:2in2}. Let $Y\in\lines$ be arbitrary. 
  If $Y\in\aflines$ we take 
  arbitrary point $a \in Y$, 
  a cycle $C_1$  through $a$, 
  the cone $Z$ with base $C_1$ and generator $Y$,
  $b \in C_1$ with $b \neq a$, 
  $Y'\in\aflines$ with $b \in Y' \parallel Y$,
  $C_2\in\cykle$ with $a,b \in C_2$, $C_2 \not\subset Z$, and 
  the cone $Z'$ with base $C_2$ and generator $Y$.
  If $Y \in \projlines$, analogously, we take 
  arbitrary point $a \in Y$, 
  a cycle $C_1$ through $a$. Next we consider three-dimensional 
  subspace $V$ of $\goth P$ such that $C_1,Y\subset V$ and we take   
  the intersection  $Z$ of $V$ and $\fixlag$.
  Let $b \in C_1$, $b \neq a$, $Y'\in\projlines$ with $b \in Y'$, 
  $Y'\subset Z$ and $Y\badjac Y'$. Let $C_2\in\cykle$ with $a,b \in C_2$, 
  $C_2 \not\subset Z$. On the end we consider subspace $V'$ of $\goth P$ 
  with $C_2,Y,Y'\subset V'$ and we take  
  the intersection $Z'$ of $V'$  and $\fixlag$.
\end{proof}
\begin{lem}
  Let $Z\in\sub_2$. The following two equivalences hold:
  \begin{multline} \label{def:gras2stozki}
    Z \in \afcones \cup \projcones  \iff 
    (\exists{Y,Y'\in\lines})\big[ Y,Y'\subset Z \land Y\neq Y' \big] \Land
    \\
    (\forall{Y,Y'\in\lines})\big[Y,Y' \subset Z \land Y\neq Y' 
    \implies (\exists{Z'\in\sub_2}) [Y,Y'\subset Z'\neq Z] \big];
  \end{multline}
  \begin{multline} \label{def:gras2minko}
    Z \in \minplanes  \iff 
    (\exists{Z'\in\sub_2})(\exists{Y,Y'\in\lines})
    [Y\neq Y' \land Z\neq Z' \land Y,Y'\subset Z,Z'] \Land
    \\
    Z \notin (\afcones\cup\projcones)
  \end{multline}
  Consequently, the class of cones and the class of Minkowski planes
  contained in \fixlag \ both are definable in $\GrasSpace(1,\fixlag)$.
\end{lem}
\begin{proof}
 Let $Z \in \afcones \cup \projcones$. By elementary geometry of planes from 
 $\afcones$ and $\projcones$, there exist 
 $Y,Y'\in\lines$ such that $Y,Y'\subset Z$ and $Y\neq Y'$.
 
 Let $Y,Y'\in\lines$ and let $Y,Y'\subset Z$ and $Y\neq Y'$. Thus there 
 exists a cycle $C_1$ such that $Z$ is the cone with base $C_1$ and 
 generator $Y$. Let $a=Y\cap C_1$ and $b=Y'\cap C_1$. Thus $a\neq b$. 
 Consider a cycle $C_2$ with $a,b\in C_2$ and $C_2\not \subset Z$. Let 
 $Z'$ be the cone with base $C_2$ and generator $Y$. Then $Z'\in\sub_2$, 
 $Y,Y'\subset Z'$ and $Z'\neq Z$.

 Assume the right-hand-side of \eqref{def:gras2stozki}.
 If 
 $Z\in \mobplanes$ then there does not exist $Y\in\lines$ with $Y\subset Z$ and we 
 have a contradiction. If $Z\in \minplanes$ then there exist $Y,Y'\in\lines$ 
 such that $Y\neq Y'$, $Y,Y'\subset Z$ and $Y\not \badjac Y'$. Thus for every 
 $Z'\in\sub_2$ such that $Y,Y'\subset Z'$ we get $Z=Z'$. This contradicts our 
 assumption. If $Z\in \projplanes \cup  \safplanes \cup  \afplanes$ then 
 $Y,Y'\in\lines$ with $Y,Y'\subset Z$ and $Y\neq Y'$ uniquely determine $Z$. 
 Thus we have a contradiction again. Hence $Z \in \afcones \cup \projcones$, 
 which completes the proof of (\ref{def:gras2stozki}).

 Let $Z \in \minplanes$. Thus $Z \notin (\afcones\cup\projcones)$. 
 By elementary geometry of Minkowski plane, 
 there exist $Y,Y'\in\lines$ with $Y\neq Y'$, $Y\badjac Y'$ and 
 $Y,Y'\subset Z$. Let $c=Y\cap Y'$ and let $C_1$ be a cycle such that 
 $c\notin C_1\subset Z$. Thus there exist points $a=Y\cap C_1$, $b=Y'\cap C_1$. 
 Of course $a\neq b$. 
 Let $C_2$ be a cycle such that $a,b\in C_2$ and 
 $C_2\not \subset Z$. 
 Consider the intersection $Z'$ of \fixlag\ and the
 3 - subspace $V$ of $\goth P$ with $C_2,Y,Y'\subset V$.

 Let $Z\in\sub_2$; assume that the right-hand-side of \eqref{def:gras2minko} holds.
 %
 %
 If 
 $Z\in \mobplanes$ then there does not exist $Y\in\lines$ with $Y\subset Z$ and we 
 have a contradiction. 
 If $Z\in \projplanes \cup  \safplanes \cup  \afplanes$ then a pair
 $Y,Y'\in\lines$ with $Y,Y'\subset Z$ and $Y\neq Y'$ uniquely determines $Z$. 
 Thus we have a contradiction again. 
 Hence $Z \in \minplanes $, and we have completed the 
 proof of (\ref{def:gras2minko}).
\end{proof}

By elementary geometry of planes from $\minplanes$, $\projcones$ and $\afcones$  
we get
\begin{lem}\label{lem:gras2afproj}
  Let $Y \in \sub_1$. We have
  \begin{eqnarray} \label{def:gras2projlines}
    Y \in \izolines & \iff & Y \in\lines \Land 
    (\exists{Z \in \minplanes})[Y \subset Z];
    \\ \label{def:gras2aflines}
    Y \in \aflines & \iff & Y \in\lines \Land Y \notin \izolines.
  \end{eqnarray}
  Further, let $Z\in\sub_2$. Then
  \begin{eqnarray} \label{gras2stozki}
    Z \in \projcones & \iff & Z \in (\projcones\cup\afcones)
    \Land (\exists{Y\in\izolines})[Y \subset Z];
  \\ \label{gras2walce}
    Z \in \afcones & \iff & Z \in (\projcones\cup\afcones)
    \Land Z \notin \projcones.
  \end{eqnarray}
\end{lem}
\begin{proof}
 The right-to-left implication of \eqref{def:gras2projlines} is evident. Let 
 $Y \in \izolines$. (Of course $Y \in\lines$.) Thus there exists a base $\quadr$ of 
 $\fixlag$ with $Y\subset \quadr$. If $Q\in \minplanes$ we take $Z:=\quadr$. 
  \begin{enumerate}
   \item
    $\quadr\not \in \minplanes$. Let $V_1$ be the subspace of $\goth P$ spanned by 
    $\quadr$ and let $W_1$ be a non tangent to $Q$  hyperplane of $V_1$
    such that $Y\subset W_1$. Then  
    $Q_1 = W_1\cap \quadr$ is a ruled quadric with 
    $\dim(Q_1) = \dim(\quadr) - 1$ and $Y\subset Q_1$. If $Q_1\in \minplanes$ then we
    take $Z:= Q_1$.
   \item
    $Q_1\not \in \minplanes$. Let $V_2$ be the subspace of $\goth P$ spanned by 
        $Q_1$ and let $W_2$ be a non tangent to $Q_1$  hyperplane of $V_2$
        such that $Y\subset W_2$. 
        Then  $Q_2 = W_2\cap Q_1$ is a ruled quadric with 
        $\dim(Q_2) = \dim(Q_1) - 1$ and $Y\subset Q_2$. 
        If $Q_2\in \minplanes$ then we take $Z:=Q_2$. 
  \end{enumerate}
 If $Q_2\not \in \minplanes$ then after a finite number of steps 
 analogous to the above we find 
 $Q_i\in \minplanes$ with $Y\subset Q_i$ and we set $Z:=Q_i$. 

 Now equivalences of \eqref{def:gras2aflines}, \eqref{gras2stozki} 
 and \eqref{gras2walce} are evident.
\end{proof}
\begin{cor}
  The structure of ``conic" pencils i.e. the structure
  \begin{ctext}
    $\struct{\izolines,\projcones,\subset}$
  \end{ctext}
  is definable in $\GrasSpace(1,\fixlag)$.
\end{cor}

\subsection{``Conic" pencils}

Now, we assume that $\basdim\geq 3$.
Then for every point $a$ of $\fixlag$ there exists 
$S' \in \projcones$ with vertex $a$.  Let us pay attention to the structure
\def\ConSpace{\mbox{\boldmath${\mathscr C}$}}
\begin{ctext}
    $\ConSpace := \struct{\izolines,\projcones,\subset} \cong 
    \struct{\izolines,\conicpeki}$.
\end{ctext}

Let us begin with the following characterization of the adjacency relation
of projective lines:
\begin{lem}\label{lem:stoz2adjac}
  Let $L_1,L_2 \in\izolines$. Then
  \begin{multline}\label{def:stoz2adjac}
    L_1 \badjac L_2 \iff (\exists{S' \in \projcones})\big[ L_1,L_2 \subset S'\big] \Lor
    (\exists{M_1,M_2\in\izolines})(\exists{S_1,S_2 \in\projcones})
    \\
    \big[ L_1,M_1,M_2 \subset S_1 \land L_2,M_1,M_2 \subset S_2  \land M_1\neq M_2\big].
  \end{multline}
\end{lem}
\begin{proof}
  The right-to-left implication of \eqref{def:stoz2adjac} is evident;
  the projective lines which are contained in a cone in $\projcones$
  all pass through its vertex. 
  
  The point is to prove the converse implication.
  Let $a\inc L_1,L_2$ and let $Z$ be a $2$-dimensional subspace that contains $L_1,L_2$.

  If $Z \in\projcones$ we are done.
  
  Let $Z\in \minplanes$. Let $\quadr$ be a base of $\fixlag$ such that $L_1, L_2\subset \quadr$. 
  Consider a hyperplane $V$ of projective space spanned by $\quadr$, tangent to $\quadr$ at the 
  point $a$. Thus $V\cap \quadr$ contains a cone $S'\in\projcones$ with vertex $a$ such 
  that $L_1, L_2\subset S'$.

  Let $Z \in \projplanes \cup \safplanes$. Let $b\inc L_1$ and $b\neq a$. Thus there 
  exists a cycle $C$ with $b\inc C$. Consider $Z_1\in \sub_2$ spanned by $L_1$ and $C$. 
  Thus $Z_1\subset \projcones\cup \minplanes$. By above, 
  there exists a cone 
  $S_1 \in \projcones$, with vertex $a$, such that $L_1\subset S_1$. 
  Let a cycle $C_1$ be a base of the cone $S_1$ with $b\inc C_1$. 
  Let $c , d\inc C_1$ with $\neq (c, d, b)$. 
  Thus there exist $M_1, M_2\in \projlines$ such that $a, c\inc M_1$, $a, d\inc M_2$. 
  Of course $M_1, M_2\subset S_1$. Consider the lines $M_1, M_2$, and $L_2$. 
  Note
  that no plane of $\goth P$ contains $M_1\cup M_2\cup L_2$. Hence 
  $M_1$, $M_2$, and $L_2$ together span a $3$-space $V_1$ of $\goth P$. Thus 
  $S\cap V_1\in \sub_2$. Now it is easily seen that $S\cap V_1\in \projcones$. 
  Let $S_2 = S\cap V_1$. Hence we 
  proved  the right-hand-side of \eqref{def:stoz2adjac}.
\end{proof}
From \ref{lem:stoz2adjac} we learn that the adjacency of lines is definable 
in terms of the geometry of $\ConSpace$. To prove that the whole geometry 
of \fixlag \ can be interpreted in the geometry of $\ConSpace$ it suffices to prove that
the concurrency of  lines in $\izolines$ is definable in terms of their
adjacency. 
\begin{lem}\label{lem:prostatniestozek}
  Let $L_1,L_2$ be two distinct projective lines contained in a cone 
  $S'\in\projcones$ with vertex $a$.
  If $L\in \lines$ and $L \badjac L_1,L_2$ then $a \inc L$.
\end{lem}
\begin{proof}
  Let $L\in \lines$ and  $L \badjac L_1,L_2$. Suppose that $a\not \inc L$. 
  Then there exist two points $b, c$ such that $b\neq c$, \  $b, c\inc L$, \ 
  $b\inc L_1$, and $c\inc L_2$. But, by elementary geometry of planes from $\projcones$,  
  there exists a cycle $C$ with $b,c\inc C$. 
  Hence we have a contradiction.
\end{proof}
\begin{lem}\label{lem:stoz2points}
  Let $S' \in\projcones$, $a$ be the vertex of $S'$, and $L \in\projlines$.
  Then 
  \begin{equation}\label{def:stoz2points}
    a \inc L \iff (\forall{L'\in\projlines})[L'\subset S' \implies L\badjac L'\Lor L=L'].
  \end{equation}
\end{lem}
\begin{proof}
  Let $a \inc L$. Assume that $L'\in\projlines$ and $L'\subset S'$. Thus $a \inc L'$, 
  so $L\badjac L' \Lor L=L'$.

  Assume the right-hand-side of \eqref{def:stoz2points}.
  Consider 
  $L'_1, L'_2\in\projlines$ with $L'_1, L'_2\subset S'$ and $L'_1\neq L'_2$. 
  From assumption 
  we get $L\badjac L'_1, L'_2 \lor L=L'_1\lor L=L'_2$. 
  By \ref{lem:prostatniestozek}, $a \inc L$.
\end{proof}

In view of \ref{lem:stoz2points} the family 
\begin{ctext}
  $\left\{ \{ L\in\projlines\colon  
  (\forall{L'\in\projlines})[L'\subset S' \implies L\badjac L'\Lor L=L']\} 
  \colon S' \in\projcones \right\}$
\end{ctext}
coincides with the family 
  $\{ \wiazka(a)\cap \projlines \colon a \in S\}$
and thus the latter is definable in \ConSpace.
Since, clearly,
$\struct{S,\projlines,\in} \cong 
  \struct{{\{\wiazka(a)\cap \projlines \colon a \in S\}},\projlines,\ni}$,
we conclude with
\begin{cor}
  The structure $\struct{S,\projlines}$ is definable in \ConSpace.
\end{cor}

Note also that in terms of the adjacency of projective lines we can distinguish two
cases: {\em adjacent lines  are on a cone} and 
{\em adjacent lines are on a (affine, semiaffine, or projective) plane}.
Note that the case when $L_1,L_2$ ($L_1\badjac L_2$) are on a cone may cover also the case 
when they are on a Minkowski plane contained in \fixlag.
\begin{prop}\label{prop:adjac2cone}
  Let $L_1,L_2\in\projlines$, $L_1 \badjac L_2$.
  The following conditions are equivalent
  \begin{itemize}\def\labelitemi{--}
  \item
    There is $S'\in\projcones$ such that $L_1,L_2 \subset S'.$
  \item
    The formula
    \begin{equation}\label{def:adjac2cone}
       (\forall{K_1,K_2\in\projlines})
       [L_1,L_2 \badjac K_1,K_2 \implies 
       K_1 \badjac K_2\Lor K_1=K_2]
    \end{equation}
    holds.
  \end{itemize}
\end{prop}
\begin{proof}
  Assume that there is $S'\in\projcones$ such that $L_1,L_2 \subset S'$. 
  Let $K_1,K_2\in\projlines$ and $L_1,L_2 \badjac K_1,K_2$. 
  By  \ref{lem:prostatniestozek}, $a\inc K_1, K_2$ where $a$ is 
  the vertex of $S'$. 
  Hence $K_1 \badjac K_2\lor K_1=K_2$. 

  Assume \eqref{def:adjac2cone}.
  Let $Z\in\sub_2$ with $L_1,L_2 \subset Z$. 
  Thus either $Z\in\projcones$ or $Z\in \projplanes$ or 
  $Z\in \safplanes$ or $Z\in \minplanes$. 
  Evidently, if $Z\in \minplanes$ then 
  there exists $S'\in\projcones$ with $L_1,L_2 \subset S'$. 
  However if 
  $Z\in \projplanes$ or $Z\in \safplanes$ then $Z$ contains 
  $K_1,K_2$ which contradicts \eqref{def:adjac2cone}. 
  Finally, there is $S'\in\projcones$ such that $L_1,L_2 \subset S'$.
\end{proof}

\subsection{Planar pencils}

Usually, when one deals with structures with lines then he  considers {\em planar}
pencils. 
Consequently, one should primarily consider the structure 
\begin{enumerate}[a)]\itemsep-2pt
\item\label{planpenc1}
   ${\goth G} = \struct{\projlines,\projpeki}$, 
\item\label{planpenc2}
  ${\goth G} = \struct{\projlines,\projpeki\cup\propsafpeki}$, 
  and 
\item\label{planpenc3}
  ${\goth G} = \struct{\lines,\peki_{\lines}}$,
  where 
  $\peki_{\lines} = 
  \left\{ \{ L\in\lines\colon a \in L \subset Z \}\colon 
  a \in Z \in \safplanes\cup\afplanes\cup\projplanes \right\}$.
\end{enumerate}
The obtained structures are partial linear spaces.
Here some standard methods used in Grassmann geometries of polar and projective
spaces can be used (cf. \cite{cohen}, \cite{pankow}; 
in any case we begin with determining maximal cliques of
the collinearity of the considered structure and maximal strong 
subspaces.

\subsubsection{Case \ref{planpenc1})}

Now the maximal cliques of $\goth G$ fall into two classes:
\begin{enumerate}[(\ref{planpenc1}.1)]\itemsep-2pt
\item\label{plp1:typ1}
  $\left\{\topof(Z) \colon Z \in \projplanes \right\}$,
  where $\topof(Z) = \{ L\in\projlines \colon L\subset Z \}$,
\item\label{plp1:typ2}
  $\left\{ [a,Y]\colon a \in Y\in \subgen^0_{\gendim} \right\}$,
  where $[a,Y] = \{ L\in\projlines\colon a \in L \subset Y \}$.
\end{enumerate}
Simultaneously, they are maximal strong subspaces of $\goth G$;
actually, they are projective spaces. 
Within the partial linear space
$\goth G$ we have $\dim(\topof(Z)) = 2$ and $\dim([a,Y]) = \gendim-1$.
If $\gendim \neq 3$ then the above two types of maximal cliques can be
distinguished in terms of the geometry of $\goth G$.

\subsubsection{Case \ref{planpenc2})}

The maximal cliques of $\goth G$ (and, at the same time, maximal strong subspaces of $\goth G$) 
fall into two classes:
\begin{enumerate}[(\ref{planpenc2}.1)]\itemsep-2pt
\item\label{plp2:typ1}
  $\left\{\topof(Z) \colon Z \in \projplanes\cup\safplanes \right\}$,
  where $\topof(Z)$ is as in (\ref{planpenc1}.\ref{plp1:typ1}),
\item\label{plp2:typ2}
  $\left\{ [a,Y]\colon a 
  \in Y\in \subgen^{\afdim}_{\gendim} \right\}$,
  $[a,Y]$ is as in (\ref{planpenc1}.\ref{plp1:typ2}).
\end{enumerate}
Note that the subspace $\topof(Z)$ carries the geometry of 
projective plane (when $Z \in\projplanes$) or the geometry of affine
plane (when $Z \in \safplanes$). In any case, from the point of view of $\goth G$,
$\dim(\topof(Z)) = 2$. Thus these two types of subspaces 
in the class \ref{plp2:typ1} are distinguishable within $\goth G$.

In $\goth G$ we have $\dim([a,Y]) = \dim(Y)-1 = \gendim+\afdim-1$;
thus maximal strong subspaces of the form $[a,Y]$ have dimension $2$
only in the case when $\gendim+\afdim = 3$.
If $\gendim = 2$, $\afdim=1$ then $Y \in \subgen^1_2$ and 
$[a,Y]$ is a semiaffine plane (a projective plane with one point deleted)
and thus a subspace of the form $[a,Y]$ is distinguishable from subspaces
of the form $\topof(Z)$.

If $\gendim = 1$, $\afdim = 2$ then $[a,Y]$ is an affine plane.
At the same time, $\projplanes = \emptyset$ and thus {\em every} maximal 
subspace of $\goth G$ is an affine plane and there is no way to distinguish
the two types (\ref{planpenc2}.\ref{plp2:typ1}), (\ref{planpenc2}.\ref{plp2:typ2})
following the above way.
In any other case these two types are distinguishable. 
Finally, for $Y_1,Y_2 \in \subgen^{\afdim}_{\gendim}$ and points $a_i\in Y_i$
it suffices to characterize the relation 
$a_1 = a_2$ in terms of the subspaces 
${\cal X}_1 = [a_1,Y_1]$, ${\cal X}_2 = [a_2,Y_2]$ and the geometry of $\goth G$
to get that the structure $\struct{S,\projlines}$ can be defined in $\goth G$.

\subsubsection{Case \ref{planpenc3})}

In this case the maximal cliques of $\goth G$ fall into two classes:
\begin{enumerate}[(\ref{planpenc3}.1)]\itemsep-2pt
\item\label{plp3:typ1}
  subsets of $\topof^\ast(Z) = \left\{ L\in\lines \colon L \subset Z \right\}$, where
  $Z \in \projplanes\cup\safplanes\cup\afplanes$, and
\item\label{plp3:typ2}
  $\left\{ [a,Y]^\ast\colon a \in Y \in \subgen_{\gendim}^{\afdim} \right\}$,
  where $[a,Y]^\ast = \{ L\in\lines\colon a \in L \subset Y \}$.
\end{enumerate}
If $Z \in \projplanes$ then $\topof^\ast(Z)$ is a clique.
If  $Z \in \afplanes$ then $\topof^\ast(Z)$ is not any clique as it contains
a pair of parallel lines; a clique $\cal K$ contained in $\topof^\ast(Z)$
is a selector of the horizon of $Z$:
it contains exactly one line in each of the directions on $Z$.
If $Z\in\safplanes$ then a clique contained in $\topof^\ast(Z)$
has form $\topof(Z) \cup \{ L \}$, where $L$ is an affine line on $Z$
and $\topof(Z)$ is defined in (\ref{planpenc1}.\ref{plp1:typ1}).
Note that cliques of type \ref{plp3:typ1} are subspaces of $\goth G$
only when $Z \in\projplanes$ or $Z \in \afplanes$ and a clique $\cal K$ 
contained in $\topof^\ast(Z)$ is a selector of the horizon of $Z$ which is a pencil, 
and cliques of the form $[a,Y]^\ast$ are subspaces of $\goth G$. If ${\cal K}$ 
is a clique contained in  $\topof^\ast(Z)$ then $\cal K$ spans in $\goth G$ the 
subspace $\topof^\ast(Z)$. As previously, if $\dim(Y)\neq 3$ then the class of 
cliques $[a,Y]^\ast$ is distinguishable and then we can reconstruct 
$\struct{S,\lines}$ in $\goth G$.

\subsubsection{Pencils of affine lines}

Note that the structure 
${\goth G} = \struct{\aflines,\peki_{\aflines}}$
with
$\peki_{\aflines} = \left\{ \{ L\in\aflines\colon a \in L \subset Z \}
\colon a \in Z \in \afplanes \right\}$
is useless to characterize the geometry of $\fixlag$.
One can even expect that under some additional assumptions the structure
$\struct{S,\aflines}$ can be defined in $\goth G$; clearly,
$\goth G$ is definable in $\struct{S,\aflines}$,
but $\fixlag$
is not definable in the latter.


\noindent
{\small Krzysztof Radziszewski\\
 Mathematics Teaching and Distance Learning Centre of\\ Gda{\'n}sk University of Technology\\
 ul. Zwyci{\c{e}}stwa 25\\ 80-233 Gda{\'n}sk\\ Poland\\ 
e-mail: krzradzi@pg.gda.pl}

\end{document}